\documentclass[12pt,a4paper]{article}

\usepackage[a4paper,margin=1in]{geometry}
\usepackage[T1]{fontenc}
\usepackage[utf8]{inputenc}
\usepackage{amsmath,amssymb,amsthm,mathtools}
\usepackage{graphicx}
\usepackage{float}
\usepackage{caption}
\usepackage{authblk}
\usepackage[colorlinks=true,linkcolor=blue,citecolor=blue,urlcolor=blue]{hyperref}

\numberwithin{equation}{section}

\newtheorem{theorem}{Theorem}
\newtheorem{lemma}{Lemma}
\newtheorem{definition}{Definition}
\newtheorem*{ektheorem}{Theorem (Edelman--Kostlan)}
\theoremstyle{remark}
\newtheorem*{remark}{Remark}

\DeclareMathOperator{\erf}{erf}

\newcommand{\Prob}{\mathbb{P}}
\newcommand{\Exp}{\mathbb{E}}

\title{\bfseries Probabilistic Zero Bounds of Certain Random Polynomials}

\author[1]{Sajad Ahmad Sheikh}
\author[1]{Mohammad Ibrahim Mir}
\affil[1]{Department of Mathematics, University of Kashmir, South Campus, Anantnag 192101, Jammu and Kashmir, India}
\affil{\href{mailto:sajadsheikh@uok.edu.in}{\texttt{sajadsheikh@uok.edu.in}}; \href{mailto:ibrahimmath80@gmail.com}{\texttt{ibrahimmath80@gmail.com}}}

\date{}

\begin{document}

\maketitle

\begin{abstract}
This paper introduces the notion of probabilistic zero bounds for random polynomials. It presents new results regarding the probabilistic bounds of random polynomials whose coefficients are independently and identically distributed as standard normal variates. Additionally, the paper provides a clear exposition of the developed methodology. To establish our results, we develop a novel approach utilizing the classical Cauchy's bounds for the zeros of a deterministic polynomial with complex coefficients. We also corroborate our analytical results with extensive simulations. The methodology developed in the paper can potentially be applied to a broad class of problems regarding bounds and the distribution of zeros in the theory of random polynomials.
\end{abstract}


\section{Introduction}

The field of random polynomials, which originated from the seminal work of Bloch and Polya in 1932, has since developed into a thriving area of research \cite{Bloch1932,BharuchaReidSambandham1986}. Their pioneering contributions paved the way for foundational investigations by Littlewood and Offord \cite{BharuchaReidSambandham1986,LittlewoodOfford1938,Hammersley1956}. In the 1940s, Kac devised a formula, now referred to as Kac's formula, which stands as one of the most significant outcomes in the field \cite{Kac1943}. This formula furnishes an asymptotic expression for the expected number of real roots of high-degree random polynomials. The topic of the distribution of zeros of random polynomials has been extensively studied, with contributions coming from works such as \cite{LittlewoodOfford1938,Hammersley1956,Bloch1932,ErdosTuran1950,Rice1939,EdelmanKostlan1995}, as well as the detailed exposition by Bharucha-Reid and Sambandham \cite{BharuchaReidSambandham1986}. As per Kac's formula, the expected number of real zeros of a degree $n$ polynomial with independent standard normal coefficients is
\begin{equation*}
\begin{aligned}
\Exp N_n(\mathbb{R})
&=\frac{1}{\pi}\int_{-\infty}^{\infty}
\sqrt{\frac{1}{(t^2-1)^2}-\frac{(n+1)^2t^{2n}}{(t^{2n+2}-1)^2}}  dt \\
&=\frac{4}{\pi}\int_0^1
\sqrt{\frac{1}{(1-t^2)^2}-\frac{(n+1)^2t^{2n}}{(1-t^{2n+2})^2}} dt .
\end{aligned}
\end{equation*}
Kac \cite{Kac1943} further showed that
\begin{equation*}
\Exp N_n(\mathbb{R})\sim \frac{2}{\pi}\log n .
\end{equation*}
Several researchers have sharpened Kac's original estimate; one can see \cite{BharuchaReidSambandham1986} for a rigorous treatment of this estimate.

In their elegant paper, Edelman and Kostlan \cite{EdelmanKostlan1995} provide an elementary geometric derivation of the Kac integral formula, similar to Buffon's needle problem, for the expected number of real zeros of a random polynomial with independent standard normally distributed coefficients. When the coefficients of the polynomial are independent and conform to a standard normal distribution, Edelman and Kostlan \cite{Kostlan1993} proved the following.

\begin{ektheorem}
Let $P_n$ be the Kac polynomial of degree $n$. Then, as $n\to\infty$, the expected number of real zeros of the polynomial is given by
\begin{equation}
\Exp N_n(\mathbb{R})=\frac{2}{\pi}\log(n)+0.6257358072\ldots+\frac{2}{n\pi}+O\left(\frac{1}{n^2}\right).
\end{equation}
\end{ektheorem}

While the distribution of roots for large-degree polynomials and the concentration phenomena have been extensively investigated in \cite{BharuchaReidSambandham1986,Rice1939,EdelmanKostlan1995,SheppVanderbei1995,ThangrajSambandham2021,DobribanKabluchko2011}, the probabilistic bounds for finite and low-degree polynomials remain largely unexplored. In the study of random polynomials, it is of significant interest to determine the likelihood that the roots of such polynomials are confined within a certain bound. More formally, consider the ensemble $\mathcal{P}_n$ of random polynomials of the form
\begin{equation}
P(z)=\sum_{k=0}^{n}a_k z^k.
\end{equation}
where the coefficients $a_k$ adhere to a given probability distribution. With respect to the uniform probability measure $d\mu_P$ on $\mathcal{P}_n$, we seek to ascertain lower bounds on the probability that a given positive constant $c$ serves as a bound for the roots of a polynomial randomly selected from $\mathcal{P}_n$. That is, we aim to determine the probability $p$ such that the roots of the polynomial $P(z)$ lie within the disc $|z|\leq c$ with at least probability $p$. Computing this exact probability is arguably a formidable problem and, to make a start, we propose and present a method to compute bounds for this probability.

\medskip
\noindent\textbf{Method and Motivation.}
Given a positive number $c$ and a random polynomial with independent and identically distributed standard Gaussian coefficients, the question of determining the probability that $c$ will act as the upper bound for the zeros may appear simple, especially for large $n$ in light of the circular laws \cite{SheppVanderbei1995}, which guarantee that for polynomials of high degree with independently and identically distributed coefficients the zeros tend to cluster around the unit circle with high probability. However, the convergence is probabilistic and the distribution of the zeros with the largest modulus is a different question. Moreover, the determination of that probability for finite $n$ is an absolutely non-trivial problem. However, it is possible to establish probabilistic bounds for this occurrence. These probabilistic zero bounds are potentially of crucial significance for understanding the typical root locations of finite-degree random polynomials as well as providing asymptotic root bounds for polynomials of high degree. In this paper, we introduce new findings on the probabilistic bounds for the roots of random polynomials whose coefficients are distributed according to the standard normal distribution by applying Cauchy's classic root bounds for complex polynomials. To that end, we propose an approach that can be adapted to a broad range of problems about the probabilistic bounds for roots of random polynomials. More specifically, the method consists in starting with any known deterministic bound of a polynomial, which generally happens to be some function of the coefficients, say $f(a_0,a_1,\ldots,a_n)$. For a random polynomial, the quantity $f(a_0,a_1,\ldots,a_n)$ is a random variable. The essence of the method is to use probabilistic methods and inference to elicit information about the location of zeros of the polynomial. It is worth noting that the proposed method can be adapted to a wide variety of problems in polynomial theory, and despite its straightforward adaptability it has the potential to deliver significant results to some otherwise very difficult problems regarding the location of zeros of random polynomials.

The remainder of the article is structured as follows. In Section~\ref{sec:prelim}, we present prerequisites in the form of definitions and lemmas. The main results regarding the probabilistic bounds for roots of polynomials with coefficients distributed as standard normal variables are presented in Section~\ref{sec:main}. In Section~\ref{sec:numerics}, we discuss the numerical simulations of the obtained results. This is followed by the conclusion in Section~\ref{sec:conclusion} and finally the suggested future work in Section~\ref{sec:future}.

\section{Some Preliminaries and Definitions}\label{sec:prelim}

In this section, we introduce the notion of probabilistic bound along with a formal definition and present a definition related to the transitivity of random variables as well as two lemmas for subsequent developments.

\begin{definition}[Probabilistic bound]
For a class $\mathcal{P}_n$ of random polynomials
\begin{equation*}
P(z)=\sum_{k=0}^{n}a_k z^k
\end{equation*}
with a prescribed distribution on its coefficients and degree $n$, let $d\mu_P$ be the uniform probability measure on $\mathcal{P}_n$. A positive number $c$ is a probabilistic bound for the roots with probability $p$ if the roots of a randomly picked polynomial in $\mathcal{P}_n$ with probability measure $d\mu_P$ lie within the disc $|z|\leq c$ with probability $p$.
\end{definition}

\begin{lemma}[Probabilistic transitivity]\label{lem:transitivity}
If $X(\omega)$ and $Y(\omega)$ are random variables such that $X(\omega)\leq Y(\omega)$ and the distribution function of $Y$ is $F(y)$, then for any positive number $c$,
\begin{equation*}
Y\leq c \quad \Longrightarrow \quad \Prob(X\leq c)\geq F(c).
\end{equation*}
\end{lemma}

\begin{proof}
Since $X(\omega)\leq Y(\omega)\leq c$, we have
\begin{equation*}
{Y(\omega)\leq c}\subset {X(\omega)\leq c}.
\end{equation*}
Hence,
\begin{equation*}
\Prob(X(\omega)\leq c)\geq \Prob(Y(\omega)\leq c).
\end{equation*}
As $\Prob(Y(\omega)\leq c)$ is precisely equal to $F(c)$, the lemma is established.
\end{proof}

\begin{lemma}[Cauchy's theorem for a monic polynomial]\label{lem:cauchy-monic}
The polynomial
\begin{equation*}
P(z)=z^n+a_{n-1}z^{n-1}+\cdots+a_0
\end{equation*}
has all its roots in the disc
\begin{equation}
|z|\leq 1+M, \qquad M=\max_{0\leq j\leq n-1}|a_j|.
\end{equation}
\end{lemma}

\begin{lemma}[Cauchy's theorem for a general polynomial]\label{lem:cauchy-general}
The roots of the polynomial
\begin{equation*}
P(z)=\sum_{i=0}^{n}a_i z^i, \qquad a_n\neq 0,
\end{equation*}
lie in the disc
\begin{equation}
|z|\leq 1+\max_{0\leq i\leq n-1}\left\{\left|\frac{a_i}{a_n}\right|\right\}.
\end{equation}
\end{lemma}
For the proof of Lemmas~\ref{lem:cauchy-monic} and~\ref{lem:cauchy-general}, one may consult \cite{RahmanSchmeisser2002}. For more details on the probabilistic preliminaries and distribution theory, one may consult any standard book on probability theory, such as \cite{Ross2014,Feller2008}.

\section{Main Results}\label{sec:main}

In the following theorems, we present results regarding the probabilistic bounds for roots of monic polynomials with coefficients distributed as standard normal variables.
\begin{theorem}\label{thm:monic-lower}
Consider the ensemble $\mathcal{P}_n$ of random polynomials
\begin{equation*}
P(z)=z^n+a_{n-1}z^{n-1}+\cdots+a_0
\end{equation*}
of degree $n$, and let $d\mu_P$ be the uniform probability measure on $\mathcal{P}_n$, where $a_i\sim N(0,1)$ are independently distributed standard normal random variables. Then a real number $c>1$ is a probabilistic bound for the roots of the polynomial $P(z)$, with probability at least $p$, if
\begin{equation}
c>\sqrt{2}\,\erf^{-1}\left(p^{1/n}\right)+1,
\end{equation}
where $\erf^{-1}$ is the inverse error function.
\end{theorem}

\begin{proof}
Consider a polynomial $P(z)=z^n+a_{n-1}z^{n-1}+\cdots+a_0$ chosen with measure $d\mu_P$ from the ensemble $\mathcal{P}_n$, where $a_i\sim N(0,1)$ are independently distributed standard normal random variables. By Lemma~\ref{lem:cauchy-monic}, the polynomial has all its roots in the disc $|z|\leq 1+M$, where $M=\max_{0\leq j\leq n-1}|a_j|$.

Denote by $F_X(x)$ the distribution function of $a_k\sim N(0,1)$ and the distribution function of $Y=|a_k|$ by $F_Y(y)$. By the definition of the cumulative distribution function, for $y\geq0$,
\begin{align}
F_Y(y)&=\Prob(Y\leq y)\notag\\
&=\Prob(-y\leq a_k\leq y)\notag\\
&=\frac{1}{\sqrt{2\pi}}\int_{-y}^{y}e^{-t^2/2}\,dt\notag\\
&=F_X(y)-F_X(-y).
\end{align}
If $X\sim N(0,1)$, then the CDF $F_X(x)$ in terms of the error function is given by
\begin{equation}
F_X(x)=\frac{1}{2}\left(1+\erf\left(\frac{x}{\sqrt{2}}\right)\right),
\end{equation}
where the error function is
\begin{equation*}
\erf(x)=\frac{2}{\sqrt{\pi}}\int_0^x e^{-t^2}\,dt.
\end{equation*}
It follows that the distribution of $Y=|X|$ is given by
\begin{align}
F_Y(y)&=\frac{1}{2}\left[\left(1+\erf\left(\frac{y}{\sqrt{2}}\right)\right)-\left(1+\erf\left(\frac{-y}{\sqrt{2}}\right)\right)\right]\notag\\
&=\erf\left(\frac{y}{\sqrt{2}}\right).
\end{align}
Hence the distribution of $V=1+\max_{0\leq j\leq n-1}|a_j|$ is
\begin{equation}
F_V(y)=\left(F_Y(y-1)\right)^n
=\left(\erf\left(\frac{y-1}{\sqrt{2}}\right)\right)^n,
\qquad y\geq 1.
\end{equation}
Therefore, if $p$ is the probability that the roots of $P(z)$ are all less than a number $c$, then
\begin{equation}
p\geq \left(\erf\left(\frac{c-1}{\sqrt{2}}\right)\right)^n.
\end{equation}
Hence $c$ is a bound with probability at least $p$ if
\begin{equation}
\erf\left(\frac{c-1}{\sqrt{2}}\right)\geq p^{1/n}.
\end{equation}
Equivalently,
\begin{equation*}
c\geq \sqrt{2}\,\erf^{-1}\left(p^{1/n}\right)+1.
\end{equation*}
This completes the proof.
\end{proof}
\begin{remark}
Setting $p=0.99$ in Theorem~\ref{thm:monic-lower}, we see that $6.32$ will be a bound of the roots of $P(z)$ of degree $100000$, with at least $99$ percent probability.
\end{remark}

\begin{theorem}\label{thm:monic-upper}
Consider the ensemble $\mathcal{P}_n$ of random polynomials
\begin{equation*}
P(z)=z^n+a_{n-1}z^{n-1}+\cdots+a_0
\end{equation*}
of degree $n$, where the coefficients $a_i\sim N(0,1)$, for $i=0,1,\ldots,n-1$, are independently distributed standard normal random variables. Further, let $d\mu_P$ be the uniform probability measure on $\mathcal{P}_n$. The roots of a randomly selected polynomial $P(z)$ from this ensemble lie within the disc $|z|\leq c$ with probability at most
\begin{equation}
\erf\left(\frac{c^n}{\sqrt{2}}\right).
\end{equation}
\end{theorem}

\begin{proof}
For the polynomial $P(z)=z^n+a_{n-1}z^{n-1}+\cdots+a_0$ chosen from the ensemble $\mathcal{P}_n$ with uniform measure $d\mu_P$, let $\xi_1,\xi_2,\ldots,\xi_n$ be the $n$ roots of $P(z)$. From Vieta's formula,
\begin{equation*}
|\xi_1\xi_2\cdots\xi_n|=|a_0|.
\end{equation*}
If $\xi_{\max}=\max_{1\leq i\leq n}|\xi_i|$, then
\begin{equation*}
\xi_{\max}^n\geq |\xi_1\xi_2\cdots\xi_n|=|a_0|,
\end{equation*}
which leads to
\begin{equation*}
\xi_{\max}\geq |a_0|^{1/n}.
\end{equation*}
For the random variable $W=|a_0|^{1/n}$, the distribution function $F_W(w)$ of $W$ is given by
\begin{equation}
F_W(w)=\Prob\left(|a_0|^{1/n}\leq w\right)=\Prob\left(|a_0|\leq w^n\right).
\end{equation}
Using the expression for the cumulative distribution function of the modulus of a standard normal variate, we obtain
\begin{equation}
F_W(w)=\erf\left(\frac{w^n}{\sqrt{2}}\right).
\end{equation}
Since $\xi_{\max}\geq W$, we have
\begin{equation}
\Prob(\xi_{\max}\leq c)\leq F_W(c).
\end{equation}
Hence
\begin{equation*}
\Prob(\xi_{\max}\leq c)\leq \erf\left(\frac{c^n}{\sqrt{2}}\right).
\end{equation*}
Therefore, the probability of all roots of $P(z)$ lying in $|z|\leq c$ is at most $\erf(c^n/\sqrt{2})$. This completes the proof.
\end{proof}

\begin{remark}
Theorem~\ref{thm:monic-lower} gives lower bounds for the probability that all the roots of $P(z)$ lie within $|z|\leq c$, while Theorem~\ref{thm:monic-upper} provides upper bounds for the same probability, thus giving a better idea about the probabilistic bounds for roots of monic polynomials with coefficients distributed as standard normal variables.
\end{remark}

In the next theorem, we obtain probabilistic bounds for roots of general polynomials with coefficients distributed as standard normal variables.

\begin{theorem}\label{thm:general}
Consider the ensemble $\mathcal{Q}_n$ of random polynomials
\begin{equation*}
Q(z)=\sum_{i=0}^{n}a_i z^i
\end{equation*}
of degree $n$, where the coefficients $a_i\sim N(0,1)$, for $i=0,1,\ldots,n$, are independently distributed standard normal random variables, and let $d\mu_Q$ be the uniform probability measure on $\mathcal{Q}_n$. A positive real number $c>1$ serves as a probabilistic bound for the roots of a randomly selected polynomial $Q(z)$ from this ensemble with probability at least $p$ if
\begin{equation}
c\geq 1+\tan\left(\frac{\pi}{2}p^{1/n}\right).
\end{equation}
\end{theorem}

Before proceeding with the proof of the theorem, we establish the following lemma.

\begin{lemma}[Distribution of the modulus of the ratio of two standard normal variables]\label{lem:ratio}
Suppose $X_1$ and $X_2$ are two independently and identically distributed standard normal random variables and
\begin{equation*}
W=\left|\frac{X_1}{X_2}\right|.
\end{equation*}
Then $W$ has CDF
\begin{equation}
F_W(w)=\frac{2}{\pi}\tan^{-1}(w), \qquad w\geq0.
\end{equation}
\end{lemma}

\begin{proof}
Let $X_1,X_2\sim N(0,1)$ be independent and set
\begin{equation*}
W=\left|\frac{X_1}{X_2}\right|.
\end{equation*}
The random variable $Y=X_1/X_2$, being the ratio of two standard normal variates, follows a standard Cauchy distribution with CDF
\begin{equation}
F_Y(y)=\frac{1}{\pi}\tan^{-1}(y)+\frac{1}{2}.
\end{equation}
Hence the distribution of $W=|Y|$ can be computed as
\begin{align*}
F_W(w)&=\Prob(W\leq w)\\
&=\Prob(|Y|\leq w)\\
&=\Prob(-w\leq Y\leq w)\\
&=F_Y(w)-F_Y(-w)\\
&=\frac{1}{\pi}\left[\tan^{-1}(w)-\tan^{-1}(-w)\right]\\
&=\frac{2}{\pi}\tan^{-1}(w),
\end{align*}
which establishes the lemma.
\end{proof}

\begin{proof}[Proof of Theorem~\ref{thm:general}]
For the polynomial $Q(z)=\sum_{i=0}^{n}a_i z^i$, with $a_i\sim N(0,1)$, Lemma~\ref{lem:cauchy-general} gives the random Cauchy bound
\begin{equation*}
Y=1+\max_{0\leq i\leq n-1}\left\{\left|\frac{a_i}{a_n}\right|\right\}.
\end{equation*}
By Lemma~\ref{lem:ratio}, the distribution of each ratio modulus is given by
\begin{equation*}
F_W(w)=\frac{2}{\pi}\tan^{-1}(w).
\end{equation*}
Thus the induced bound is evaluated through
\begin{equation*}
F_Y(y)=\left[\frac{2}{\pi}\tan^{-1}(y-1)\right]^n.
\end{equation*}
In view of Lemma~\ref{lem:transitivity}, a number $c$ is a bound with probability at least
\begin{equation*}
p\geq \left[\frac{2}{\pi}\tan^{-1}(c-1)\right]^n.
\end{equation*}
Equivalently,
\begin{equation*}
\frac{2}{\pi}\tan^{-1}(c-1)\geq p^{1/n},
\end{equation*}
or
\begin{equation*}
\tan^{-1}(c-1)\geq \frac{\pi}{2}p^{1/n}.
\end{equation*}
This leads to
\begin{equation*}
c\geq 1+\tan\left(\frac{\pi}{2}p^{1/n}\right).
\end{equation*}
This completes the proof.
\end{proof}

\begin{remark}
Substituting $n=5$ and $p=0.99$ in Theorem~\ref{thm:general}, we find that $13.91$ acts as a bound for the roots of polynomial $Q(z)$ with probability at least $0.99$.
\end{remark}
\section{Discussion of Numerical Simulation Results}\label{sec:numerics}

A numerical study of the formula in Theorem~\ref{thm:monic-lower} reveals that for a monic random polynomial of degree $5$ with coefficients as independently and identically distributed standard random variables, $c=2$ is a bound with probability at least $14$ percent. While the actual probability of $91$ percent makes that a poor estimate, we also observe from a comparison of the graphs in Figs.~\ref{fig:P-theory} and~\ref{fig:P-monte-carlo} that the lower bound on the probability of $c$ to bound the roots of $P(z)$ improves very rapidly as we increase the value of $c$. For $c=3$, the lower bound of probability from Theorem~\ref{thm:monic-lower} is $62$ percent, while the actual probability computed from simulation is $98$ percent, which is a sharp improvement. Taking $c=6.32$, we find from Theorem~\ref{thm:monic-lower} that $c=6.32$ is a bound of the roots with probability at least $99$ percent. Numerical simulations also reveal that the average of lower and upper bounds given by Theorem~\ref{thm:monic-lower} and Theorem~\ref{thm:general} is a much better estimate for the actual probability computed from simulations. It is also interesting to mention that the lower bounds of probability by Theorem~\ref{thm:monic-upper} give much better estimates for the actual probability in comparison to those given by Theorem~\ref{thm:monic-lower} for $c=2$.

Figures~\ref{fig:Q-theory} and~\ref{fig:Q-monte-carlo} depict the probability of a number being a bound of $Q(z)$ from Theorem~\ref{thm:general} and Monte Carlo simulations, respectively, while Figs.~\ref{fig:P-max-root} and~\ref{fig:Q-max-root} show the distribution of the root with maximum modulus for random polynomials $P(z)$ and $Q(z)$. The distribution of the zeros of a random non-monic polynomial with independent and identically distributed normal coefficients with degree $1000$ has been depicted in Fig.~\ref{fig:unit-circle}. Finally, Fig.~\ref{fig:expected-max-cauchy} shows the expected modulus of the largest root for polynomials of degrees $1$ through $20$, represented by the red line with circle markers. The expected Cauchy bound for these polynomials is represented by the blue line with square markers. The expected modulus of the largest root tends to increase with the degree of the polynomial. This trend suggests that higher-degree polynomials are more likely to have roots with larger moduli. The expected Cauchy bound, which is a theoretical upper bound on the modulus of all roots, also increases with the degree of the polynomial. However, it appears to grow at a faster rate than the actual expected modulus of the largest root, which indicates that while the Cauchy bound is a useful estimate, it is not tight. Though the Cauchy bound serves as a ready bound for the zeros of a polynomial, for more precise estimations or sharper bounds, one would likely need to rely on numerical simulations or more refined theoretical bounds tailored to the specific properties of the coefficients of the polynomials.

\begin{figure}[H]
\centering
\begin{minipage}[b]{0.45\textwidth}
\centering
\includegraphics[width=0.95\linewidth]{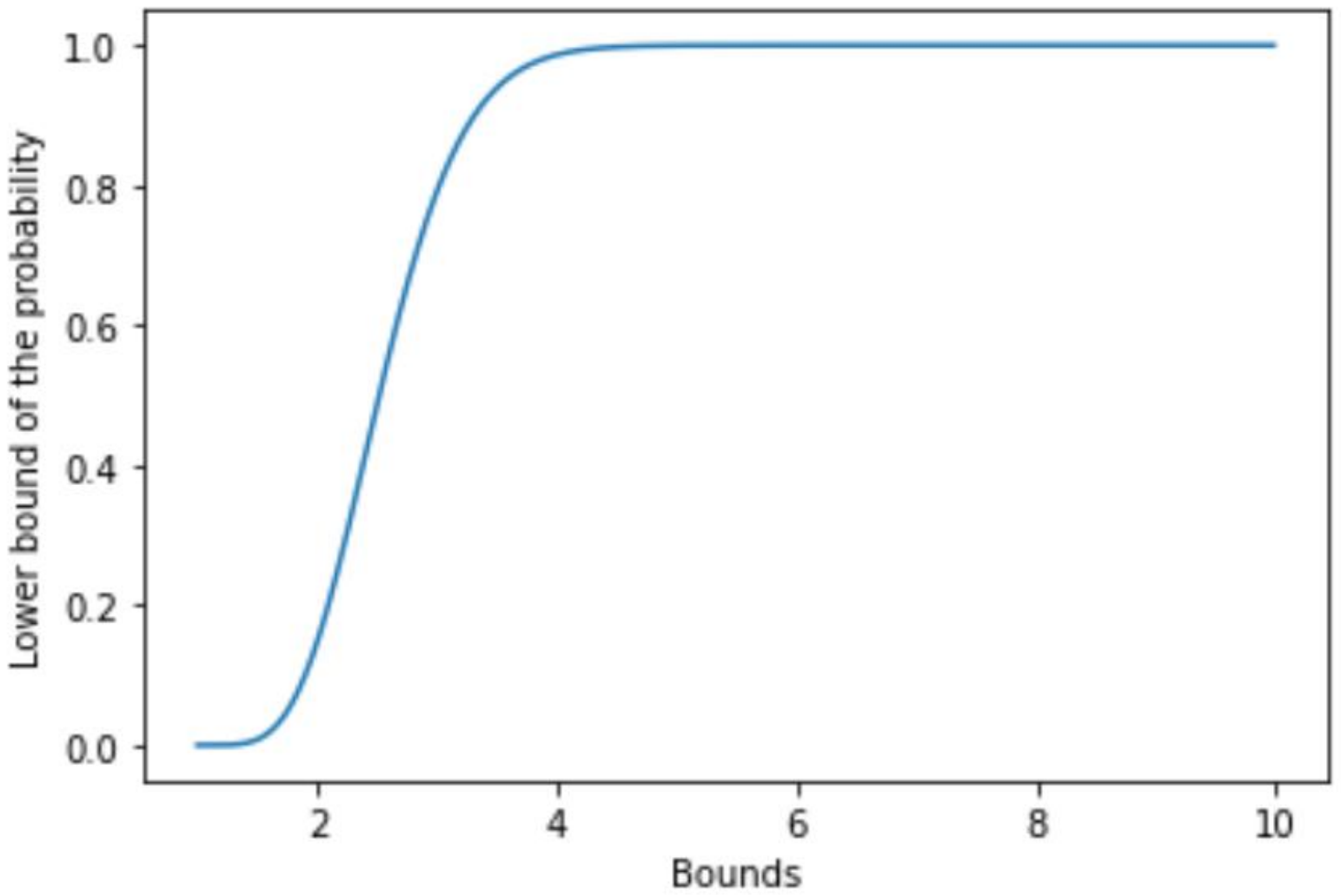}
\caption{Simulation of probabilistic bounds for roots of $P(z)$ for $n=5$ by Theorem~\ref{thm:monic-lower}.}
\label{fig:P-theory}
\end{minipage}\hfill
\begin{minipage}[b]{0.45\textwidth}
\centering
\includegraphics[width=0.95\linewidth]{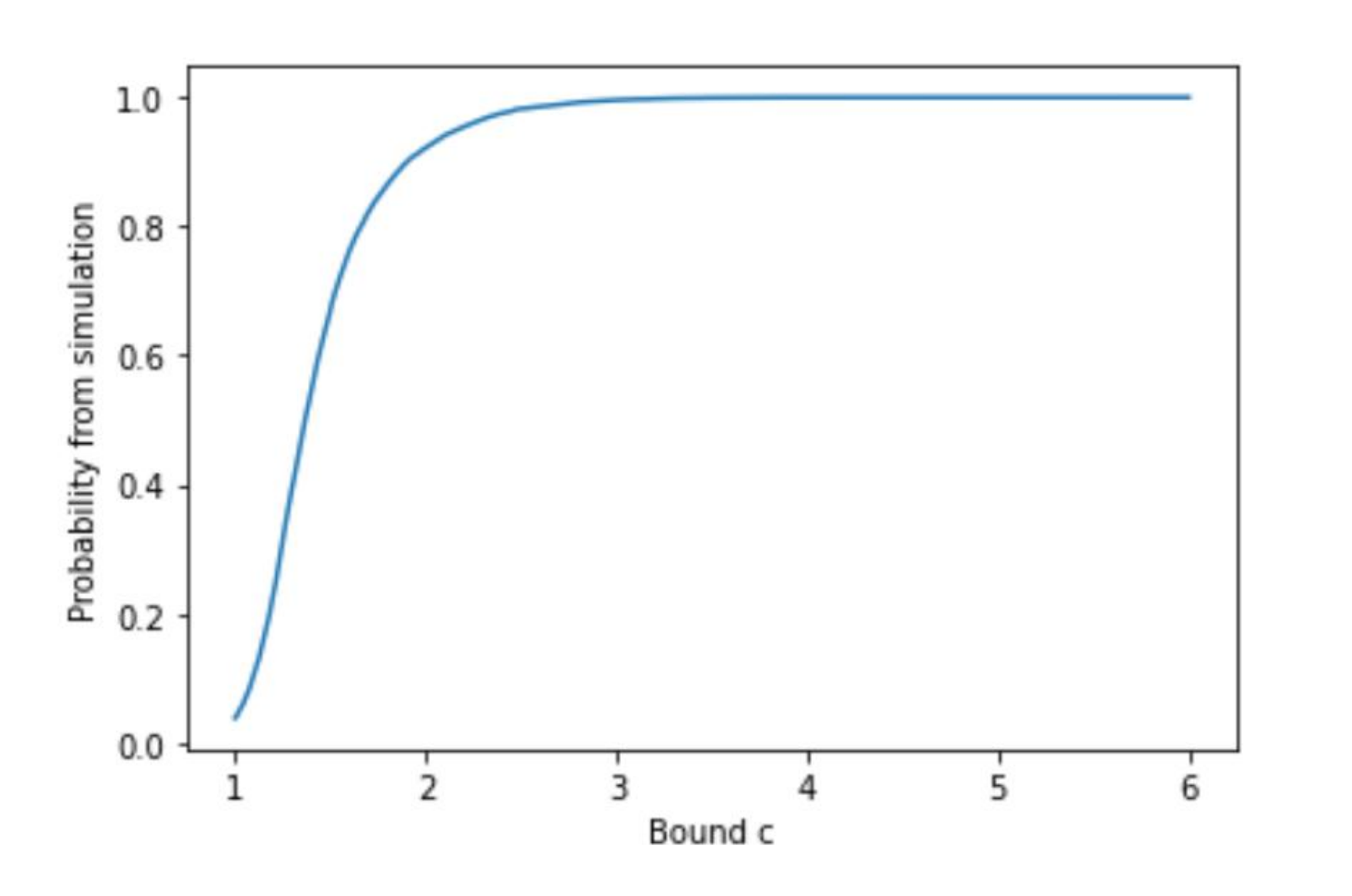}
\caption{Probabilistic bounds for roots of $P(z)$ for $n=5$ by Monte Carlo simulation.}
\label{fig:P-monte-carlo}
\end{minipage}
\end{figure}

\begin{figure}[H]
\centering
\begin{minipage}[b]{0.45\textwidth}
\centering
\includegraphics[width=0.95\linewidth]{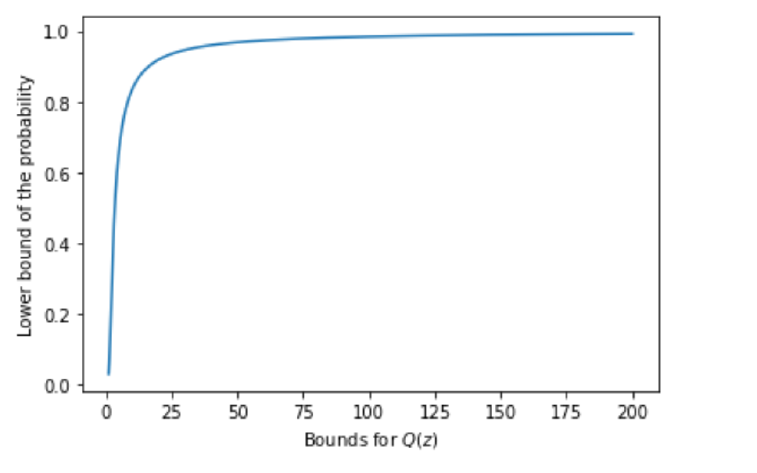}
\caption{Simulation of probabilistic bounds for roots of $Q(z)$ for $n=5$ by Theorem~\ref{thm:general}.}
\label{fig:Q-theory}
\end{minipage}\hfill
\begin{minipage}[b]{0.45\textwidth}
\centering
\includegraphics[width=0.95\linewidth]{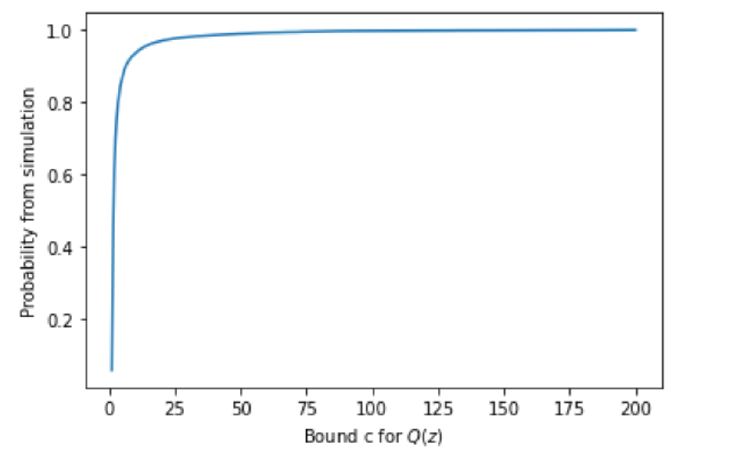}
\caption{Probabilistic bounds for roots of $Q(z)$ for $n=5$ by Monte Carlo simulation.}
\label{fig:Q-monte-carlo}
\end{minipage}
\end{figure}

\begin{figure}[H]
\centering
\begin{minipage}[b]{0.45\textwidth}
\centering
\includegraphics[width=0.95\linewidth]{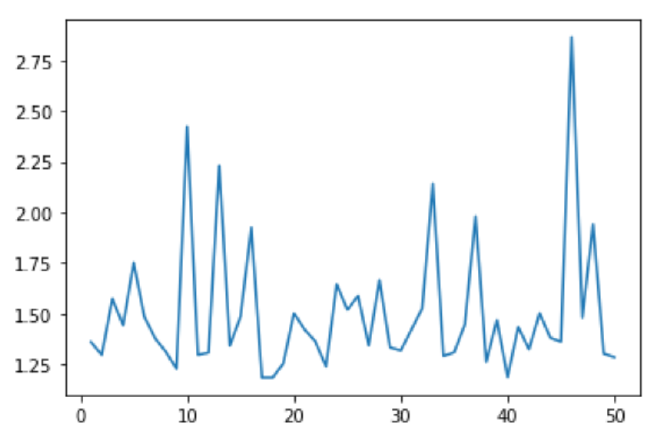}
\caption{Distribution of the root with maximum modulus for $P(z)$.}
\label{fig:P-max-root}
\end{minipage}\hfill
\begin{minipage}[b]{0.45\textwidth}
\centering
\includegraphics[width=0.95\linewidth]{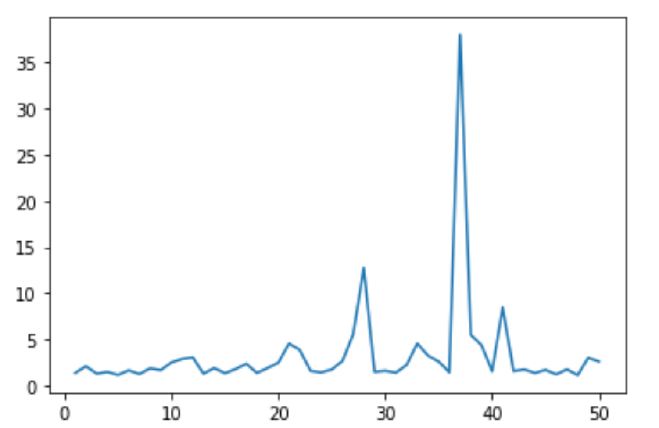}
\caption{Distribution of the root with maximum modulus for $Q(z)$.}
\label{fig:Q-max-root}
\end{minipage}
\end{figure}

\begin{figure}[H]
\centering
\begin{minipage}[b]{0.45\textwidth}
\centering
\includegraphics[width=0.95\linewidth]{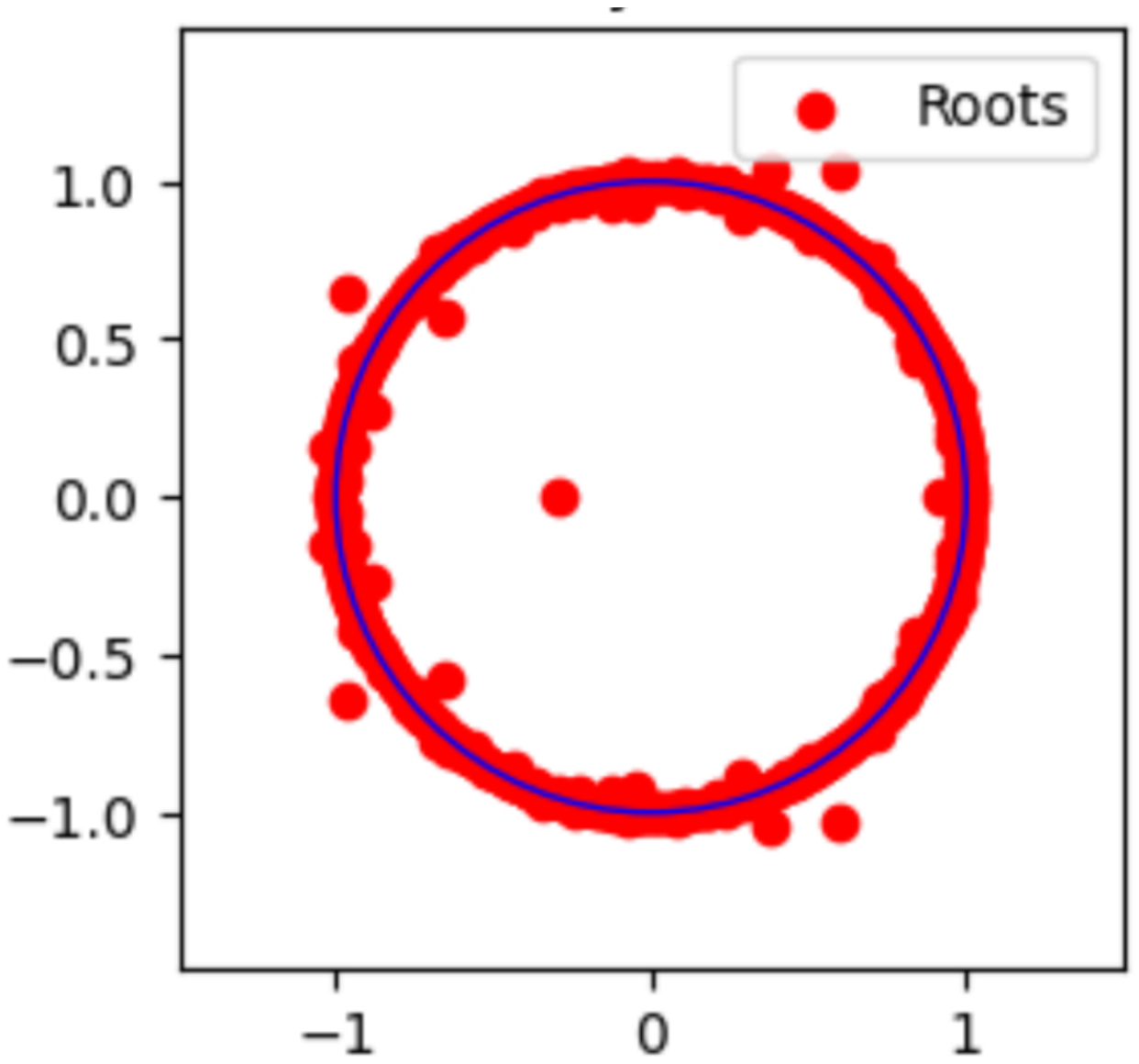}
\caption{Distribution of the roots with respect to the unit circle for $Q(z)$.}
\label{fig:unit-circle}
\end{minipage}\hfill
\begin{minipage}[b]{0.45\textwidth}
\centering
\includegraphics[width=0.95\linewidth]{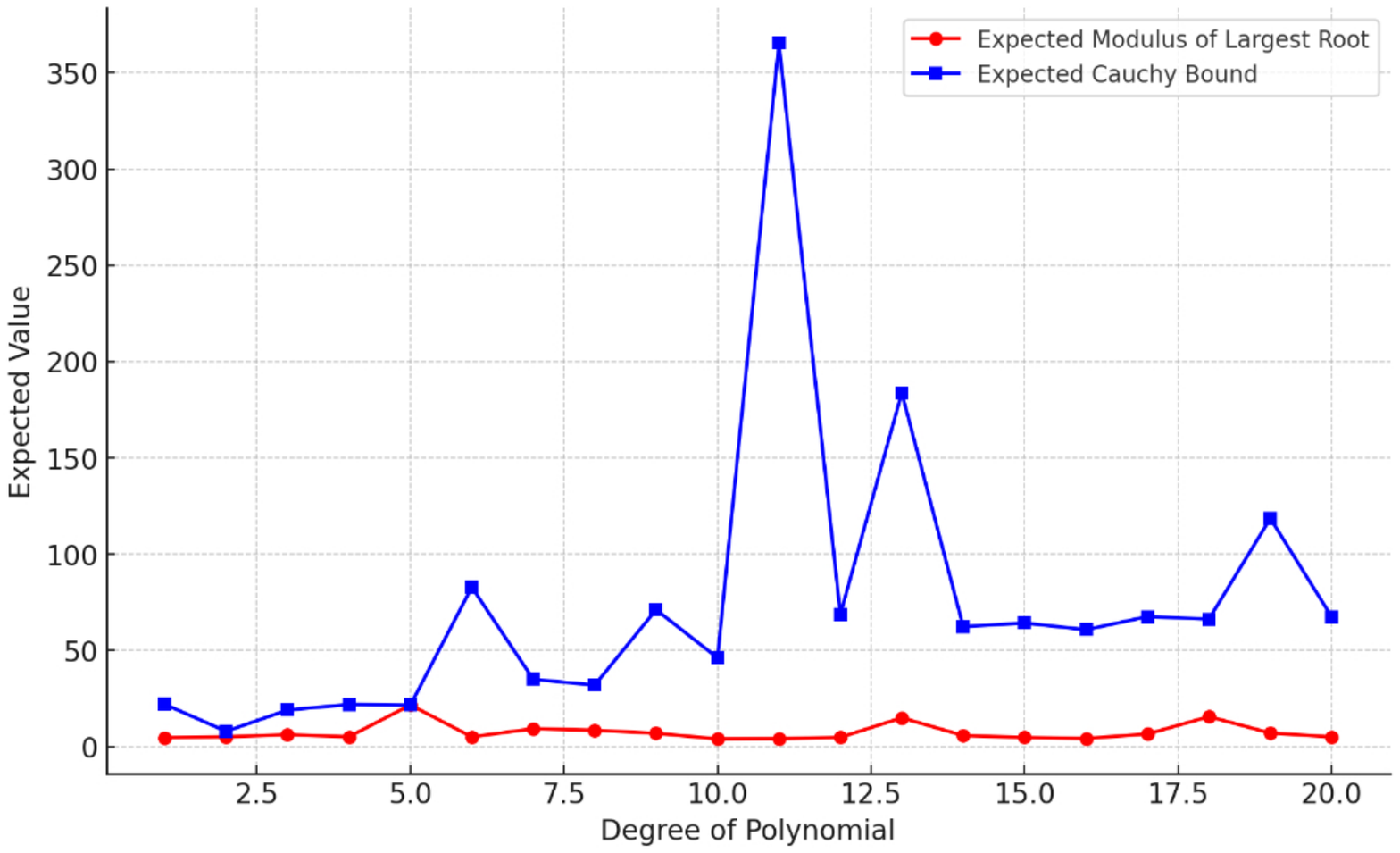}
\caption{Distribution of the expected maximum modulus and Cauchy bound for $Q(z)$.}
\label{fig:expected-max-cauchy}
\end{minipage}
\end{figure}

\section{Conclusion}\label{sec:conclusion}

In this paper, we introduce the notion of probabilistic zero bounds and, through a novel yet simple method, illustrate its application by deriving some interesting theorems utilizing the classical bounds for zeros of a polynomial due to Cauchy. Our investigation into the probabilistic bounds of roots of random polynomials with independent and identically distributed standard normal coefficients has yielded concrete results that contribute to the existing body of knowledge on the distribution of zeros in random polynomials, exemplifying how classical mathematical concepts, such as Cauchy's theorem, can be leveraged in a probabilistic framework to yield insights into the behavior of complex random polynomials. The expected number of zeros within the unit disk is derived, adding to our understanding of the distribution of zeros for polynomials with random coefficients through an elegant symmetry argument. The methodological approach adopted in this paper---utilizing Cauchy's theorem and probabilistic logic to infer the locations of zeros---is noteworthy for its simplicity and effectiveness. This approach could potentially be generalized to a wider class of problems in polynomial theory, offering a robust toolset for tackling questions about root locations in random polynomial settings.

\section{Future Work}\label{sec:future}

The future work stemming from this paper's methodology opens several promising avenues for exploration. As a method paper that innovates upon the classical bounds for the zeros of polynomials, we envision extending the probabilistic zero-bounds approach to other well-known bounds beyond those of Cauchy. Several classical bounds, such as the ones established by Kalantari \cite{Kalantari2005}, Dehmer and Mowshowitz \cite{DehmerMowshowitz2011}, Datt and Govil \cite{DattGovil1978}, and a plethora of other bounds offer different perspectives on the locations of polynomial zeros. Applying the probabilistic framework developed in this paper to these bounds could yield new insights into the distribution of zeros for random polynomials with various distributions of coefficients.

\section*{Conflict of Interests}
The authors have no conflict of interest to declare.

\section*{Data Availability}
The authors declare that they have not used any third-party or copyrighted data in their simulations and all the programs used can be made available on request.

\end{document}